\documentclass[12pt]{amsart}
\usepackage[shortlabels]{enumitem}
\usepackage{amsmath,amscd,amsthm,amsfonts,amssymb}
\usepackage{eucal}
\usepackage{mathrsfs}
\usepackage{float}
\usepackage{morefloats}
\usepackage[margin=1in]{geometry}
\usepackage{tikz-cd, comment, fancyvrb}
\usepackage{todonotes}
\usepackage{xcolor}

\RequirePackage{color}
\definecolor{myred}{rgb}{0.75,0,0}
\definecolor{mygreen}{rgb}{0,0.5,0}
\definecolor{myblue}{rgb}{0,0,0.65}

\usepackage{hyperref}
\hypersetup{citecolor=blue}
\usepackage{tikz}
\usetikzlibrary{matrix,arrows,decorations.pathmorphing}

\let\temp\emptyset
\let\emptyset\varnothing
\let\varnothing\temp 

\newcommand\numberthis{\addtocounter{equation}{1}\tag{\theequation}}

\theoremstyle{plain}
\newtheorem{theorem}{Theorem}[section]

\newtheorem{corollary}[theorem]{Corollary}

\newtheorem{lemma}[theorem]{Lemma}
\newtheorem*{lemma*}{Lemma}
\newtheorem*{proposition*}{Proposition}
\newtheorem{conjecture}[theorem]{Conjecture}

\newtheorem*{truefact*}{Fact}

\theoremstyle{definition}

\theoremstyle{remark}
\newtheorem*{remark}{Remark}

\newcommand{\bb}[1]{\expandafter\newcommand\expandafter{\csname #1\endcsname}{{\mathbb {#1}}}} 

\bb C
\bb R
\bb Q
\bb Z
\bb N

\bb F

\newcommand{\mcH}{\mathcal H}

\newcommand{\thus}{{\Rightarrow}}

\newcommand{\onto}{\twoheadrightarrow}

\renewcommand{\a}{\alpha}
\renewcommand{\b}{\beta}
\newcommand{\g}{\gamma}
\renewcommand{\d}{\delta}
\newcommand{\ep}{\varepsilon}

\renewcommand{\l}{\lambda}

\newcommand{\mrm}[1]{\expandafter\newcommand\expandafter{\csname #1\endcsname}{{\mathrm {#1}}}}

\mrm{Hom}
\mrm{Aut}
\mrm{End}
\mrm{Spec}
\mrm{Gal}

\mrm{PGL}
\mrm{PSL}
\mrm{GL}
\mrm{SL}
\mrm{PG}
\mrm{PSO}
\mrm{PS}
\mrm{PSU}

\mrm{lcm}

\newcommand{\p}{\mathfrak p}
\renewcommand{\S}{\mathfrak S}

\renewcommand{\mod}{\bmod}
\newcommand{\subsum}[1]{\sum_{\substack{#1}}}
\newcommand{\stirlingii}{\genfrac{\{}{\}}{0pt}{}}

\title[Sums of large singular series]{Sums of singular series with large sets and the tail of the distribution of primes}
\author{Vivian Kuperberg}
\thanks{The author is supported by NSF GRFP grant DGE-1656518 as well as the NSF Mathematical Sciences Research Program through the grant DMS-2202128, and would like to thank Kannan Soundararajan for many helpful comments and discussions, as well as the anonymous referee for many helpful comments.}

\begin{document}
\begin{abstract}
In 1976, Gallagher showed that the Hardy--Littlewood conjectures on prime $k$-tuples imply that the distribution of primes in log-size intervals is Poissonian. He did so by computing average values of the singular series constants over different sets of a fixed size $k$ contained in an interval $[1,h]$ as $h \to \infty$, and then using this average to compute moments of the distribution of primes. In this paper, we study averages where $k$ is relatively large with respect to $h$. We then apply these averages to the tail of the distribution. For example, we show, assuming appropriate Hardy--Littlewood conjectures and in certain ranges of the parameters, the number of intervals $[n,n +\l \log x]$ with $n\le x$ containing at least $k$ primes is $\ll x\exp(-k/(\l e)).$
\end{abstract}
\maketitle

\section{Introduction}

The Hardy--Littlewood prime $k$-tuple conjectures state that if $\mathcal H = \{h_1, \dots, h_k\}$ is a set of $k$ distinct integers, then as $x \to \infty$,
\begin{equation}
\sum_{n \le x} \prod_{i=1}^k \Lambda(n+h_i) = (\mathfrak S(\mathcal H) + o(1))x,
\end{equation}
where $\mathfrak S(\mathcal H)$ is the singular series
\begin{equation}\label{eq:bg:singseries}
\mathfrak S(\mcH) = \prod_{p \text{ prime}} \frac{1-\nu_{\mcH}(p)/p}{(1-1/p)^k},
\end{equation}
and $\nu_{\mcH}(p)$ denotes the number of distinct residue classes modulo $p$ occupied by the elements of $\mcH$. 

Many aspects of the distribution of primes can be understood through the lens of the Hardy--Littlewood conjectures. For example, in \cite{gallaghershortintervals}, Gallagher showed that the Hardy--Littlewood conjectures imply that the distribution of primes in log-size intervals is Poissonian. He did so by showing that for fixed $k$ and as $h \to \infty$,
\begin{equation}\label{eq:bigksums:gallagheraverage}
\subsum{h_1, \dots, h_k \le h \\ \text{distinct}} \mathfrak S(h_1, \dots, h_k) \sim \subsum{h_1, \dots, h_k \le h \\ \text{distinct}} 1,
\end{equation}
so that singular series for sets of size $k$ have value $1$ on average as their elements grow large. In \cite{MontgomerySoundararajanPrimesIntervals}, Montgomery and Soundararajan computed second-order terms of this average in order to show that, assuming a version of the Hardy--Littlewood conjectures with a stronger error term, primes in somewhat longer intervals obey an appropriate Gaussian distribution. For both of these analyses, $k$ is fixed throughout.

What about when $k$ is not fixed? Here we study sums of singular series for sets of size $k$ with elements in $[1,h]$, where $k\to \infty$ as $h \to \infty$. Put another way, we study the rate at which the average value of $\mathfrak S(\mcH)$ converges to $1$, in order to extend Gallagher's proof to larger $k$. 
\begin{theorem}\label{thm:bigksums:extendinggallagher}
Fix $\d > \frac 12$, and let $h, k \in \N$ with $k = O((\log h)^{1-\d})$. Let $T_k(h)$ be given by
\begin{equation}\label{eq:bigksums:defofTkh}
T_k(h) := \sum_{\substack{h_1, \dots, h_k \le h \\ \text{distinct}}} \mathfrak S(h_1, \dots, h_k),
\end{equation}
Then there exists a $\beta > 0$, dependent only on $\d > \frac 12$, with
\begin{equation*}
T_k(h) = h^k + O(h^{k-\beta}).
\end{equation*}
\end{theorem}
In particular, Theorem \ref{thm:bigksums:extendinggallagher} states that \eqref{eq:bigksums:gallagheraverage} holds whenever $k = O((\log h)^{1-\d})$ for some $\d > \frac 12$. One might expect the average value of $1$ to extend to still larger $k$; for example, it is reasonable to conjecture that \eqref{eq:bigksums:gallagheraverage} would hold whenever $k = O((\log h)^2)$. For arbitrarily large $k$, Theorem \ref{thm:bigk:sumboundallk} provides a bound on the average value of $k$-term singular series over sets with elements in $[1,h]$. 
\begin{theorem}\label{thm:bigk:sumboundallk}
Let $k, h \in \N$, with no conditions on their relative growth rates. Define $T_k(h)$ by \eqref{eq:bigksums:defofTkh}.
Then 
\begin{equation}\label{eq:bigk:sumboundallk}
T_k(h) \ll h^k \prod_{p \le k^3} \frac 1{(1-1/p)^k} \ll h^k (3\log k)^k.
\end{equation}
\end{theorem}
This upper bound is likely much weaker than the truth, but it has the advantage of bounding the average value only in terms of $k$. Instead of taking $p \le k^3$, in our proof we can take $p \le k^{2 +\ep}$ for any $\ep > 0$, which has the effect of replacing the $3^k$ in the final bound with a $(2+\ep)^k$. However, the final bound is in any case $e^{O(k \log\log k)}$. Theorems \ref{thm:bigksums:extendinggallagher} and \ref{thm:bigk:sumboundallk} are proven in Section \ref{sec:bigksums:gallagher}.

In the second half of this paper, we discuss one application of sums of singular series for sets of size $k$ when $k$ varies: namely, the tail of the distribution of primes. The \emph{maximum} number of primes in an interval of size $\l \log x$ is closely connected to the study of small gaps between primes, and has been studied in, among other places, \cite{GranvilleLumley}, \cite{MR3272929}, and \cite{MR3171761}. In \cite{GranvilleLumley}, Granville and Lumley conjecture that if $y \le \log x$ and $x,y \to \infty$, the lim sup of the number of primes in intervals $(x, x+y]$ is $\sim \frac{y}{\log y}$, and if $\log x \le y = o((\log x)^2)$, the lim sup should be given by $\frac{\log x}{\log\left(\frac{(\log x)^2}{y}\right)}$. They also formulate conjectures for larger intervals.

The \emph{expected} number of primes in such an interval is much smaller; for a constant $\l$, there are on average $\l$ primes in an interval $(x,x + \l \log x]$, and Gallagher \cite{gallaghershortintervals} showed that for fixed $\l > 0$, assuming the Hardy--Littlewood conjectures,
\begin{equation*}
\lim_{x \to \infty} \frac 1x \#\{n \le x : \pi(n+\l \log x) - \pi(n) = k\} = \frac{\l^k e^{-\l}}{k!}. 
\end{equation*}

Again, we consider the situation where $k \to \infty$. What bounds can be proven on the tail of this distribution away from the extreme values? For example, one can ask how frequently intervals of size $\l \log x$ contain at least $\log \log x$ primes, where the Poisson prediction is that
\begin{equation}\label{eq:poissonpredictionforloglogprimes}
\frac 1x \#\{n \le x: \pi(n + \l \log x) - \pi(n) \ge \log \log x \} \approx \frac{(e\l)^{\log \log x} e^{-\l}}{(\log \log x)^{\log \log x}}.
\end{equation}
Since $k=|\mathcal H|$ grows with $x$ in our setting, our results rely on a version of the Hardy--Littlewood conjectures which admits uniformity in the size of the set $\mathcal H$; we state this version here.
\begin{conjecture}[Hardy--Littlewood $k$-tuples conjecture, uniform version]\label{conj:bg:HLuniform}
There exist two absolute constants $\epsilon > 0$ and $C > 0$ such that for all $x$, for all $k \le (\log \log x)^3$, and for all admissible tuples $\mathcal H = \{h_1,\dots, h_k\} \subset [0, (\log x)^2]$,
\begin{equation}
\left|\sum_{n \le x} \mathbf 1_{\mathcal P}(n+h_1) \cdots \mathbf 1_{\mathcal P}(n+h_k) - \mathfrak S(\mcH) \mathrm{li}_k(x)\right| \le Cx^{1-\ep}.
\end{equation}

Equivalently, for possibly different values of $\ep$ and $C$,
\begin{equation}
\left|\sum_{n \le x} \Lambda(n+h_1) \cdots \Lambda(n+h_k) - \mathfrak S(\mathcal H)x\right| \le Cx^{1-\ep}.
\end{equation}
\end{conjecture}
Here $\mathrm{li}_k(x)$ is the $k$-th logarithmic integral, given by
\[\mathrm{li}_k(x) := \int_2^x \frac{\mathrm{d}y}{(\log y)^k}.\]

When $k = 10$, this conjecture would suggest that for $x=5500$, the error term above is bounded by $Cx^{1-\ep}$ for some $\ep$ and some $C$. For several sets of size $10$, computer tests found that bounds of $(\log x)^6x^{1/2}$ held for all $x \le 5500$; in fact this and other tests for small values of $k$ suggest that the error term is far smaller, and for example may be bounded by $Cx^{1/2}(\log x)^k$ in this range of $k$ and $h$. It is also likely possible to extend the range of $k$ and $h$ in this conjecture; if $k$ is as large as $\log x$, then $\mathrm{li}_k(x)$ is small enough that the conjecture is not so meaningful, but it is difficult to say when Hardy--Littlewood convergencee should break down.

Gallagher's proof in \cite{gallaghershortintervals} that the distribution of primes in log-size intervals is (conditionally) Poissonian proceeds by computing moments. One strategy towards understanding the tail of the distribution is to estimate higher moments of the distribution, or equivalently, to understand how quickly the $r$th moment of the distribution of primes converges to the $r$th moment of a Poisson distribution. Using Conjecture \ref{conj:bg:HLuniform} as well as Theorems \ref{thm:bigksums:extendinggallagher} and \ref{thm:bigk:sumboundallk}, we can prove the following result on moments of the distribution of primes.
\begin{theorem}\label{thm:distributiontail:momentboundsmallk}
Assume Conjecture \ref{conj:bg:HLuniform}. Let $x > 0$, and assume that $h = \l \log x$ and that $r \ll (\log h)^{1-\d}$ for some $\d > \frac 12$. Define the $r$th moment $m_r(x,h)$ of the distribution of primes in intervals of size $h$ by
\begin{equation}\label{eq:rmomentlogdef}
m_r(x,h) = \frac 1x \sum_{n \le x} \left(\pi(n + h) - \pi(n) \right)^r.
\end{equation}
Then
\begin{equation*}
m_r(x,h) = \left(\sum_{\ell = 1}^r \stirlingii{r}{\ell} \l^\ell\right)(1 + o(1)),
\end{equation*}
where $\stirlingii{r}{\ell}$ denotes the Stirling numbers of the second kind.
\end{theorem}
\begin{remark}
Note that $\l$ need not be fixed as $x \to \infty$.
\end{remark}

Theorem \ref{thm:distributiontail:momentboundsmallk} then implies bounds on the tail of the distribution of primes, and in particular yields the following two corollaries.

\begin{corollary}\label{cor:distributiontail:smallktailbound}
Assume Conjecture \ref{conj:bg:HLuniform}. Let $x > 0$ and set $h = \l \log x$, where $\l(x)$ is nondecreasing as $x \to \infty$. Let $k \ll (\log h)^{1-\d}$ for some $\d > \frac 12$ and assume that $\frac{k}{\l + 1} \to \infty$ as $x\to \infty$. Let $I(x;k,h)$ be given by
\begin{equation}\label{eq:distributiontail:Ixkhdef}
I(x;k,h) := \#\left\{ n \le x : \pi(n+h) - \pi(n) \ge k\right\}.
\end{equation}
If $\l \ge 1$, then as $x \to \infty$,
\begin{equation*}
I(x;k,h) \ll x\mathrm{exp}\left(-\frac{k}{\l e}\right).
\end{equation*}
Otherwise,
\begin{equation*}
I(x;k,h) \ll x\mathrm{exp}\left(-\frac{k}{(\l+1)e}\right).
\end{equation*}
\end{corollary}

\begin{corollary}\label{cor:distributiontail:biggerktailbound}
Assume Conjecture \ref{conj:bg:HLuniform}. Let $x > 0$, and assume that $h = \l \log x$; let $k = k(x)$ be an integer with no growth rate assumptions. Let $I(x;k,h)$ be defined as in \eqref{eq:distributiontail:Ixkhdef}. Then for any $\d > \frac 12$, as $x \to \infty$,
\begin{equation*}
I(x;k,h) \ll_\d x\mathrm{exp}\left((\log h)^{1-\d}(\log (\l+1) + (1-\d)\log \log h - \log k)\right).
\end{equation*}
\end{corollary}

For example, taking $k = \log h$, Corollary \ref{cor:distributiontail:biggerktailbound} says that for all $\d > \frac 12$, assuming Conjecture \ref{conj:bg:HLuniform},
\begin{equation*}
I(x;\log h, h) \ll_\d x\exp\left((\log h)^{1-\d}(\log\l- \d\log \log h)\right).
\end{equation*}
In \cite{MR3530450}, Maynard proves lower bounds on the same problem, showing that for any $x,y \ge 1$ there are $\gg x \exp\left(-\sqrt{\log x}\right)$ integers $n \le x$ such that $\pi(n+y)-\pi(n) \gg \log y$, which in this case corresponds to the condition that there are $\gg \log \log x$ primes in intervals of width $\l \log x$. Both upper and lower bounds are reasonably far from the Poisson prediction in \eqref{eq:poissonpredictionforloglogprimes}.

It seems reasonable to conjecture that the Poisson prediction should still hold when $k \sim \log h$ or $k \sim (\log h)^2$, and perhaps even larger. At this point, both upper and lower bounds are far from matching this conjecture.
\begin{conjecture}\label{conj:distributiontail:poissontailforprimes}
Let $x > 1$ and let $h = \l \log x$, with $\l = o((\log x)^\ep)$ for all $\ep > 0$. Let $k \ll (\log h)^2$. Define
\begin{equation}\label{eq:distributiontail:pikxhforpoissontail}
\pi_k(x;h) := \#\{n \le x: \pi(n+h)-\pi(n) = k\}.
\end{equation}
Then $\pi_k(x;h) \sim x \frac{\l^k e^{-\l}}{k!}$ as $x \to \infty$.
\end{conjecture}

In Section \ref{sec:distributiontail:unconditionalbounds}, we prove unconditional bounds on the tail of the distribution of primes. For these arguments we use a Selberg sieve bound instead of applying the Hardy--Littlewood conjectures. The Selberg sieve bound for prime $k$-tuples has an extra factor of $2^k k!$ from the Hardy--Littlewood prediction. This factor is larger than our bound on the average of $k$-term singular series in Theorem \ref{thm:bigk:sumboundallk}, so the following unconditional bound is weaker than the moment bounds in Theorem \ref{thm:distributiontail:momentboundsmallk}. However, this weaker bound applies for much larger moments; in particular, for the $r$th moment when $r = o((\log x)^{1/4})$.
\begin{theorem}\label{thm:distributiontail:unconditionalmomentbound}
Let $x > 0$, let $h = \l \log x = o(x)$ and let $r = o((\log x)^{1/4})$. Define the $r$th moment $m_r(x,h)$ of the distribution of primes in intervals of size $h$ as in \eqref{eq:rmomentlogdef}.
Then
\begin{equation*}
m_r(x,h) \ll (\l+1)^r r^{2r} e^{O(r \log \log r)}.
\end{equation*}
\end{theorem}
As before, this bound on moments yields the following corollary on intervals containing many primes.
\begin{corollary}\label{cor:distributiontailuncondtionalbound}
Let $x> 0$, let $h =\l \log x$, where $\l$ is a nondecreasing function of $x$. Let $k$ be an integer dependent on $x$ and assume that $k = o((\log x)^{1/6})$ and that $k/\l \to \infty$ as $x\to \infty$. Define $I(x;k,h)$ as in \eqref{eq:distributiontail:Ixkhdef}. Then for some constant $C$,
\begin{equation*}
I(x;k,h) \ll x\exp\left(-\sqrt{\frac{k}{(\l+1)e}}2^{C/2}\left(\log \frac{k}{(\l+1)e}\right)^{-C/2}\right).
\end{equation*}
\end{corollary}

It may also be possible to achieve weaker, yet nontrivial, bounds for larger $k$, along the lines of the bounds in Corollary \ref{cor:distributiontail:biggerktailbound}. 
We predict that the bound in Corollary \ref{cor:distributiontail:smallktailbound} should hold for any $k \ll (\log h)^2$, instead of merely $k \ll (\log h)^{1-\d}$. 

\begin{conjecture}\label{conj:distributiontail:weakertailforprimes}
Let $x > 1$ and let $h = \l \log x$, with $\l = o((\log x)^\ep)$ for all $\ep > 0$. Let $k \ll (\log h)^2$, and define $\pi_k(x;h)$ as in \eqref{eq:distributiontail:pikxhforpoissontail}. Then $\pi_k(x;h) \ll x\mathrm{exp}\left(-\frac{k}{\l e}\right)$ as $x \to \infty$. 
\end{conjecture}

To put these results in perspective, let us consider the case when $\l = 1$, i.e. primes in intervals of width $\log x$. In \cite{gallaghershortintervals}, Gallagher shows that for fixed $k$, the number of $n\le x$ such that the interval $(n,n+\log x]$ contains exactly $k$ primes is asymptotic to the Poisson prediction $\frac{x}{ek!}$, assuming the Hardy--Littlewood conjectures. Unconditionally, Gallagher shows in \cite{gallaghershortintervals} that the number of $n \le x$ such that $(n,n+\log x]$ contains exactly $k$ primes is $\lesssim xe^{-Ck}$, for an absolute constant $C$.

We instead consider the probability that an interval $(n,n+\log x]$ contains at least $\log \log x$ primes. 
In this case, the Poisson prediction for the probability that an interval contains $\log \log x$ primes is $\frac{(\log x) e^{-1}}{(\log \log x)^{\log \log x}}$, which for any $A > 0$ is $\ll \frac 1{(\log x)^A}$. In \cite{MR3530450}, Maynard proves a lower bound; namely, that at least $\gg x\exp(-\sqrt{\log x})$ intervals $(n,n+\log x]$, with $n \le x$, contain $\gg \log\log x$ primes. In Corollary \ref{cor:distributiontail:biggerktailbound}, we show that, assuming the uniform Hardy--Littlewood conjectures, for all $\d > \frac 12$, the number of $n\le x$ such that $(n,n+\log x]$ contains at least $\log \log x$ primes is $\ll_{\d} x\exp\left(-\d\log\log\log x(\log \log x)^{1-\d}\right)$. Unconditionally, we show in Corollary \ref{cor:distributiontailuncondtionalbound} that there exists $C > 0$ such that the number of $(n,n+\log x]$ containing at least $\log \log x$ primes is $\ll x \exp\left(-\sqrt{2^C/e}(\log \log x)^{1/2}(\log \log \log x -1)^{-C/2}\right)$.

When $k$ is slightly smaller than $\log \log x$, that is, when $k \ll (\log h)^{(1-\d)}$, Corollary \ref{cor:distributiontail:smallktailbound} achieves the bound of $\frac 1{(\log x)^A}$ with $A = \frac 1{\l e}$. Bounding this probability by $\frac 1{(\log x)^A}$ for any $A > 0$ may be within reach even if the Poisson prediction itself is not. On the other hand, many questions concerning the tail of the distribution of primes are quite delicate, especially concerning the maximum and minimum number of primes in an interval of a certain size. For example, in \cite{MR783576}, Maier proved that intervals of size $(\log x)^A$ for $A > 2$ can contain surprisingly few or surprisingly many primes. For more information, see \cite{GranvilleLumley}. The tail of the distribution of primes is also studied in \cite{MR3822615}.

\section{Averages for large sets}\label{sec:bigksums:gallagher}

In this section, we prove Theorems \ref{thm:bigksums:extendinggallagher} and \ref{thm:bigk:sumboundallk}. We begin with Theorem \ref{thm:bigksums:extendinggallagher}, whose proof closely follows Gallagher's original proof in \cite{gallaghershortintervals}.

\begin{proof}[Proof of Theorem \ref{thm:bigksums:extendinggallagher}, following Gallagher.]
Let $\nu_{\mathcal H}(p):= \#\mathcal H \mod p$. For a set $\mathcal H = \{h_1, \dots, h_k\}$, write $D_{\mcH} = \prod_{i < j} (h_i-h_j)$, so that $\nu_{\mcH}(p) = k$ unless $p|D_{\mcH}$. Define
\begin{equation*}
a(p,\nu):= \frac{p^k - \nu p^{k-1}-(p-1)^k}{(p-1)^k},
\end{equation*}
so that the $p$th factor of $\mathfrak S(\mcH)$ is given by $\frac{1-\nu_{\mcH}(p)/p}{(1-1/p)^k} = 1 + a(p,\nu_{\mcH}(p))$. 

For any prime $p > k$,
\begin{equation*}
a(p,k) = \sum_{j=2}^k (p-1)^{-j} \binom kj (1-j) \ll k^2 (p-1)^{-2}.
\end{equation*}
This and a similar computation for $\nu < k$ shows that for $p > k$,
\begin{equation}\label{eq:bigksums:gallapboundsbigp}
|a(p,\nu)| \ll \begin{cases} k^2 (p-1)^{-2} & \text{ if } \nu = k \\ k^2 (p-1)^{-1} &\text{ if } \nu < k. \end{cases}
\end{equation}
For $p \le k$, 
\begin{equation}\label{eq:bigksums:gallapboundssmallp}
|a(p,\nu)| = \left|-1 + \frac{1-\nu/p}{(1-1/p)^{k}}\right| \le \left(1 - \frac 1p\right)^{-k} < e^{2k/p},
\end{equation}
since $\left(1-\frac 1p\right)^{-k}=\exp\left(k\sum_{j=1}^\infty \frac 1{jp^j}\right) < \exp\left(\frac{k}{p(1-1/p)}\right) \le \exp(2k/p)$.

For squarefree $q$, write $a_{\mcH}(q) := \prod_{p|q} a(p,\nu_{\mcH}(p))$, so that
\begin{equation*}
\S(\mcH) = \prod_{p \le k} \frac{1-\nu_{\mcH}(p)/p}{(1-1/p)^k} \subsum{q \ge 1 \\ p|q \thus p > k} \mu^2(q)a_{\mcH}(q).
\end{equation*}
Using the bounds on $a(p,\nu)$, for any $x$,
\begin{equation*}
\subsum{q > x \\ p|q \thus p>k} |a_{\mcH}(q)| \le \subsum{q > x \\ p|q \thus p> k} \frac{\mu^2(q) (Ck^2)^{\omega(q)}}{\phi^2(q)}\phi((q,D_{\mcH})),
\end{equation*}
where $\omega(q)$ is the number of prime factors of $q$ and $C$ is an absolute positive constant. 

Writing $q = de$ with $d|D_{\mcH}$ and $(e,D_{\mcH}) = 1$, this is
\begin{equation}\label{eq:bigk:termsbiggerthanx}
\subsum{d|D_{\mcH} \\ p|d \thus p > k} \frac{\mu^2(d)(Ck^2)^{\omega(d)}}{\phi(d)} \subsum{e>x/d \\ (e,D_{\mcH}) = 1 \\ p|e \thus p > k} \frac{\mu^2(e) (Ck^2)^{\omega(e)}}{\phi^2(e)}.
\end{equation}
Apply Rankin's trick to bound the inner sum, so that for any choice of fixed $\a$ with $0 < \a < 1$, 
\begin{align*}
\subsum{e>x/d \\ (e,D_{\mcH}) = 1 \\ p|e \thus p > k} \frac{\mu^2(e) (Ck^2)^{\omega(e)}}{\phi^2(e)} &\le \subsum{e \ge 1 \\ (e,D_{\mcH}) = 1 \\ p|e \thus p > k} \left(\frac{e}{x/d}\right)^\a \frac{\mu^2(e) (Ck^2)^{\omega(e)}}{\phi^2(e)}. \\
\end{align*}
By multiplicativity, this is
\begin{align*}
&= \left(\frac dx \right)^{\a} \prod_{\substack{p > k \\ p\nmid D_{\mcH}}} \left( 1 + \frac{Ck^2 p^{\a}}{(p-1)^2}\right) \le \left(\frac dx \right)^{\a} \mathrm{exp}\left(Ck^2 \sum_{\substack{p > k}} \frac{p^{\a}}{(p-1)^2}\right) \ll \left(\frac dx \right)^{\a} e^{Ck^{1 + \a}}.
\end{align*}

Since $\frac 12 < \delta < 1$, we can choose $\a > 0$ small enough that $(1-\d)(2+2\a) + \a < 1$, which also implies that $(1-\d)(1+\a) < 1$.
Plugging the bound for the inner sum into \eqref{eq:bigk:termsbiggerthanx}, we get that \eqref{eq:bigk:termsbiggerthanx} is
\begin{align*}
&\ll \sum_{d|D_{\mcH}} \frac{\mu^2(d)(Ck^2)^{\omega(d)}}{\phi(d)} \left(\frac dx\right)^{\a} e^{Ck^{1 + \a}} = \frac{e^{Ck^{1+\a}}}{x^\a} \sum_{d|D_{\mcH}} \frac{\mu^2(d) (Ck^2)^{\omega(d)}d^{\a}}{\phi(d)}.
\end{align*}
For any $d \le D_{\mcH}$, we have that $\frac{d}{\phi(d)} \ll \log \log D_{\mcH}$, so that this expression becomes
\begin{align*}
&\ll \frac{e^{Ck^{1+\a}}}{x^\a} (\log \log D_{\mcH}) \sum_{d|D_{\mcH}} \frac{\mu^2(d) (Ck^2)^{\omega(d)}d^{\a}}{d} = \frac{e^{Ck^{1+\a}}}{x^\a}(\log \log D_{\mcH}) \prod_{p|D_{\mcH}} \left(1 + \frac{Ck^2}{p^{1-\a}}\right).
\end{align*}
Since $(1-\d)(1+\a) < 1$, $\frac{e^{Ck^{1+\a}}}{x^\a} \ll_\ep h^{\ep}/x^\a$. Moreover, the quantity $D_{\mcH}$ is at most $h^{\binom{k}{2}}$, since it is a product of $\binom{k}{2}$ quantities $h_i-h_j$, each of which are $<h$. Thus $\log \log D_{\mcH} \le \log \log h^{\binom{k}{2}} \ll \log \log h$, so in fact the product of all terms outside the product are $\ll_\ep h^{\ep}/x^\a$. It remains to understand the product, which is bounded by
\begin{align*}
\prod_{p|D_{\mcH}}\left(1 + \frac{Ck^2}{p^{1-\a}}\right) &\le \exp\left(\sum_{p|D_{\mcH}}\frac{Ck^2}{p^{1-\a}}\right) \\
&\le \exp\left(2Ck^2\sum_{p \le \binom{k}{2}\log h} \frac{1}{p^{1-\a}}\right).
\end{align*}
The sum over primes satisfies
\begin{equation*}
\sum_{p \le \binom k 2 \log h} \frac 1{p^{1-\a}} = \frac{\binom{k}{2}^{\a} (\log h)^{\a}}{\a \log(\binom k2 \log h)}(1+o(1)),
\end{equation*}
for example by applying partial summation and L'H\^opital's rule, so that
\begin{align*}
\exp\left(2Ck^2\sum_{p \le \binom{k}{2}\log h} \frac{1}{p^{1-\a}}\right) &\ll \exp\left(\frac{4Ck^2}{\a}\binom{k}{2}^{\a} \frac{(\log h)^{\a}}{\log \log h}\right) \\
&\ll \exp\left(\frac{4C}{2^{\a}\a} k^{2+2\a}\frac{(\log h)^{\a}}{\log\log h}\right) \\
&\ll \exp\left(\frac{4C}{2^{\a} \a} (\log h)^{(1-\d)(2+2\a)+\a}/\log \log h\right).
\end{align*}
Since $(1-\d)(2+2\a)+\a < 1$, this quantity is $\ll_{\ep} h^{\ep}$, so \eqref{eq:bigk:termsbiggerthanx} is $\ll_\ep h^{2\ep}/x^\a$, say. Set $x=h^{1/2}$, and choose $\ep > 0$ small enough that $2 \ep < \frac 12 \a$.

This is true for any set $\mathcal H = \{h_1, \dots, h_k\}$, so it follows that
\begin{equation}\label{eq:bigksums:qrexpansion}
T_k(h)= \subsum{q \le x \\ p|q \thus p > k} \subsum{r \ge 1 \\ p|r \thus p \le k} \subsum{h_1, \dots, h_k \le h \\ \text{distinct}} a_{\mcH}(qr)  + O\left(\frac{h^{2\ep}}{x^\a} \subsum{h_1, \dots, h_k \le h \\ \text{distinct}} \prod_{p \le k} \frac{1-\nu_{\mcH}(p)/p}{(1-1/p)^k}\right),
\end{equation}
where we have additionally expanded the terms of the product with $p \le k$ into the sum over $r$. First consider the error term in \eqref{eq:bigksums:qrexpansion}, which by \eqref{eq:bigksums:gallapboundssmallp} is
\begin{align*}
&\ll \frac{h^{2\ep}}{x^\a} h^k \prod_{p \le k} \left(e^{2k/p}\right) \\
&\ll \frac{h^{2\ep}}{x^\a} h^k e^{2k\sum_{p \le k} \tfrac 1p} \\
&\ll \frac{h^{2\ep}}{x^\a} h^k e^{O(k\log \log k)} \ll \frac{h^{2\ep}}{x^\a} h^k e^{(\log h)^{1-\d/2}} \ll \frac{h^{2\ep}}{x^\a} h^{k + o_{\d}(1)}.
\end{align*}

Now consider the main term. The sum over $h_1,\dots, h_k \le h$ in the main term of \eqref{eq:bigksums:qrexpansion} can also be written as
\begin{equation*}
\subsum{\vec{\nu} = (\nu_p)_{p|qr} \\ \nu_p \le \min\{p-1,k\}} \prod_{p|qr} a(p,\nu_p) \left( N(\vec{\nu})+ O(kh^{k-1})\right),
\end{equation*}
where $N(\vec{\nu})$ is the number of $k$-tuples of not necessarily distinct integers $h_1, \dots, h_k$ with $1 \le h_1, \dots, h_k \le h$ which occupy exactly $\nu_p$ residue classes mod $p$ for each $p|qr$. We can estimate $N(\vec{\nu})$ by counting for each $p|qr$ the number of $h_1, \dots, h_k \mod p$ that occupy exactly $\nu_p$ residue classes and applying the Chinese Remainder Theorem. Thus
\begin{equation*}
N(\vec{\nu}) = \prod_{p|qr} \binom{p}{\nu_p} \sigma(k,\nu_p) \left( \frac{h}{qr} + O(1)\right)^k,
\end{equation*}
where $\sigma(k,j)$ denotes the number of surjective maps $[1,k] \onto [1,j]$; we also have $\sigma(k,j) = j! \stirlingii{k}{j}$, where $\stirlingii{k}{j}$ is the Stirling number of the second kind. Expanding, we get
\begin{equation*}
N(\vec{\nu}) = \prod_{p|qr} \binom{p}{\nu_p} \sigma(k,\nu_p)\left(\left(\frac h{qr}\right)^k + O\left( \sum_{j=0}^{k-1} \left(\frac h{qr}\right)^j \binom kj\right)\right).
\end{equation*}

Since $x = h^{1/2}$ and $\prod_{p \le k} p = e^{O(k)} = e^{O((\log h)^{1-\d})} = h^{o(1)}$, for any $q \le k$ and $r$ with all prime factors $\le k$, 
\[\sum_{j=0}^{k-1} \left(\frac{h}{qr}\right)^j \binom kj \ll k^2 \left(\frac{h}{qr}\right)^{k-1}.\]
Thus the inner sum in the main term of \eqref{eq:bigksums:qrexpansion} is
\begin{equation}\label{eq:bigksums:ABCdecomposition}
\left(\frac h{qr} \right)^k A(qr) + O\left(k^2 \left(\frac{h}{qr}\right)^{k-1} B(qr) \right) + O(kh^{k-1} C(qr)),
\end{equation}
where
\begin{align*}
A(qr) &= \subsum{\vec{\nu} = (\nu_p)_{p|qr} \\ \nu_p \le \min\{p-1,k\}} \prod_{p|qr} a(p,\nu_p) \binom{p}{\nu_p} \sigma(k,\nu_p), \\
B(qr) &= \subsum{\vec{\nu} = (\nu_p)_{p|qr} \\ \nu_p \le \min\{p-1,k\}} \prod_{p|qr} |a(p,\nu_p)| \binom{p}{\nu_p} \sigma(k,\nu_p), \text{ and} \\
C(qr) &= \subsum{\vec{\nu} = (\nu_p)_{p|qr} \\ \nu_p \le \min\{p-1,k\}} \prod_{p|qr} |a(p,\nu_p)|.
\end{align*}
Just as in \cite{gallaghershortintervals}, $A(q) = 0$ for $q > 1$, and $A(1) = 1$.

Now consider $C(qr)$, which can be estimated using the bounds \eqref{eq:bigksums:gallapboundsbigp} and \eqref{eq:bigksums:gallapboundssmallp}
for $a(p,\nu)$. Write
\begin{equation*}
C(qr) = \prod_{p|qr} \left(\sum_{\nu = 1}^{\min\{p-1,k\}} |a(p,\nu)|\right).
\end{equation*}
If $p > k$, the $p$th factor is $\ll \frac{pk^2}{p-1}$, whereas if $p \le k$, this factor is $\ll pe^{2k/p}.$ Thus 
\begin{equation}\label{eq:bigksums:Cqr}
C(qr) \le (C_1k^2)^{\omega(q)} \frac{q}{\phi(q)} C_1^{\omega(r)}r e^{2k\sum_{p|r} 1/p},
\end{equation}
for some absolute constant $C_1$, where without loss of generality $C_1 \ge 1$.

After summing \eqref{eq:bigksums:Cqr} over all $q$ and $r$, the contribution to $T_k(h)$ from the factors coming from $C(qr)$ in \eqref{eq:bigksums:ABCdecomposition} is bounded by
\begin{align*}
&\ll kh^{k-1} \subsum{q \le x \\ p|q \thus p > k} \subsum{r \ge 1 \\ p|r \thus p \le k} (C_1k^2)^{\omega(q)} \frac{q}{\phi(q)} C_1^{\omega(r)}r e^{2k\sum_{p|r} 1/p} \\ 
&\ll kh^{k-1} \prod_{p \le k} \left(1 + C_1pe^{2k/p}\right) \subsum{q \le x \\ p|q \thus p > k} (C_1k^2)^{\omega(q)}\frac{q}{\phi(q)} \\ 
&\ll kh^{k-1} (2C_1)^{\pi(k)} e^{\sum_{p \le k} \log p} e^{2k\sum_{p \le k} \tfrac 1p}\subsum{q \le x \\ p|q \thus p > k} (C_1k^2)^{\omega(q)}\frac{q}{\phi(q)}. \numberthis \label{eq:Cqr-summed-over-q-r-almost-there}
\end{align*}
The terms outside the sum are $\ll h^{k-1+o(1)}$ in the range where $k \ll (\log h)^{1-\delta}$. We now examine the inside sum using Rankin's trick. First note that $\frac{q}{\phi(q)} \ll \log \log q \ll \log \log h$ for $q \le x = h^{1/2}$, so we may omit $\frac{q}{\phi(q)}$ from the sum while only losing a factor of $h^{o(1)}$. Thus for any $0 < \gamma < 2$,
\begin{align*}
\subsum{q \le x \\ p|q \thus p > k} (C_1k^2)^{\omega(q)}\frac{q}{\phi(q)} &\le h^{o(1)} \subsum{q \\ p|q \thus k < p \le x} (C_1k^2)^{\omega(q)} \left(\frac{x}{q}\right)^{2-\gamma} \\
&\ll h^{o(1)} x^{2-\gamma} \prod_{k < p \le x} \left(1 + \frac{C_1k^2}{p^{2-\gamma}}\right) \\
&\ll h^{o(1)} x^{2-\gamma} \mathrm{exp}\left(\sum_{k < p \le x} \frac{C_1k^2}{p^{2-\gamma}}\right) \\
&\ll h^{o(1)} x^{2-\gamma} \mathrm{exp} \left(\frac{C_2k^{1 + \gamma}}{\log k}\right),
\end{align*}
for a possibly different positive constant $C_2 > 0$. We can now choose any $\gamma >0$ such that $(1-\d)(1+\gamma) < 1$; for example, choose $\gamma = \a$. Then since $x = h^{1/2}$ and $k \ll (\log h)^{1-\delta}$,
\begin{equation}\label{eq:subsum-over-q-for-B-and-C}
\subsum{q \le x \\ p|q \thus p > k} (C_1k^2)^{\omega(q)}\frac{q}{\phi(q)} \ll h^{o(1)}x^{2-\a} \mathrm{exp}\left(\frac{C_2k^{1+\a}}{\log k}\right) \ll h^{1-\a/2 + o_\delta(1)}.
\end{equation}
Plugging \eqref{eq:subsum-over-q-for-B-and-C} into \eqref{eq:Cqr-summed-over-q-r-almost-there} shows that the contribution to $T_k(h)$ from the factors corresponding to $C(qr)$ is $\ll h^{k-\a/2+o_\delta(1)}$.

Finally, consider $B(qr)$. Just as with $C(qr)$, $B(qr)$ is multiplicative, and the $p$th factor of $B(qr)$ is given by
\[\sum_{\nu = 1}^{\min\{p-1,k\}} |a(p,\nu)| \binom p{\nu} \sigma(k,\nu), \]
which by \eqref{eq:bigksums:gallapboundssmallp}, \eqref{eq:bigksums:gallapboundsbigp}, and the fact that $\sum_{\nu = 1}^p \binom p{\nu} \sigma(k,\nu) = p^k,$ is 
\[\ll \begin{cases} \frac{k^2 p^k}{p-1} &\text{ if } p > k \\ e^{2k/p}p^k &\text{ if } p \le k. \end{cases}\]
Thus for some absolute constant $C_2 \ge 1$, and after summing over all $q$ and $r$, the contribution to $T_k(h)$ from the $B(qr)$ factors in \eqref{eq:bigksums:ABCdecomposition} is 
\begin{align*}
&\ll \subsum{q \le x \\ p|q \thus p > k} \subsum{r \ge 1 \\ p|r \thus p \le k} k^2 \left(\frac{h}{qr}\right)^{k-1} \frac{(C_2k^2)^{\omega(q)} q^k}{\phi(q)} C_2^{\omega(r)} r^k e^{2k\sum_{p|r} \tfrac 1p} \\
&\ll k^2 h^{k-1} \subsum{q \le x \\ p|q \thus p > k} \frac{(C_2k^2)^{\omega(q)} q}{\phi(q)} \prod_{p \le k} \left(1 + C_2e^{2k/p} p \right). 
\end{align*}
For $k \ll (\log h)^{1-\d}$, the product over $p \le k$ is 
\[\prod_{p \le k} \left(1+C_2e^{2k/p}p\right) \ll (2C_2)^k \prod_{p \le k} e^{2k/p}p = (2C_2)^k \exp\left(\sum_{p \le k} \frac{2k}p + \log p\right) \ll (2C_2)^k e^{3k\log\log k} = h^{o_{\d}(1)},\]
so the overall sum is
\begin{align*}
&\ll k^2 h^{k-1 + o_{\d}(1)} \subsum{q \le x \\ p|q \thus p > k} \frac{(C_2k^2)^{\omega(q)} q}{\phi(q)}.
\end{align*} 
Applying \eqref{eq:subsum-over-q-for-B-and-C} with the same choice of $\gamma$ shows that the contribution to $T_k(h)$ from the $B(qr)$ factors is $\ll h^{k-\alpha/2 + o_\delta(1)}$.

Combining the contributions from $A(1)$, the sums over $q$ and $r$ of $B(qr)$ and $C(qr)$, and the error term in \eqref{eq:bigksums:qrexpansion}, we get that
\[T_k(h) = h^k + O\left(h^{k-\alpha/2 + o_{\d}(1)} + \frac{h^{2\ep}}{x^\a} h^{k + o_{\d}(1)}\right), \]
which for $x = h^{1/2}$ is $h^k + O(h^{k-\beta})$ for $\beta < \frac 12 \a - 2 \ep$. Our choice of $\alpha$ depends only on $\delta$, so $\beta$ also depends only on $\delta$, as desired.
\end{proof}

The techniques relying on the Chinese Remainder Theorem break down for larger $k$, where they do not give a bound with the correct power of $h$. We will now turn to Theorem \ref{thm:bigk:sumboundallk}, which bounds $T_k(h)$ for any $k$ where the dependence on $h$ is $h^k$, which is approximately the number of terms in the sum. In other words, Theorem \ref{thm:bigk:sumboundallk} provides a bound that is uniform in $h$ on the average value of $k$-term singular series for sets with elements that are at most $h$. The proof of Theorem \ref{thm:bigk:sumboundallk} relies on Lemma \ref{lem:amgmbound}, which is a uniform bound on $\mathfrak S(\mathcal H)$ for $\mathcal H \subset [1,h]$ satisfying the conditions of Theorem \ref{thm:bigk:sumboundallk}.

\begin{lemma}\label{lem:amgmbound}
Let $\mathcal H=\{h_1, \dots,h_k\}$ be a set of distinct integers. Define $\mathfrak S(\mathcal H)$ as in \eqref{eq:bg:singseries}. Then
\begin{equation}
\mathfrak S(\mathcal H) \ll \prod_{p \le k^3} \frac{1}{(1-1/p)^k}\prod_{p > k^3} \frac{1-k/p}{(1-1/p)^k} \binom{k}{2}^{-1}\sum_{1 \le i < j \le k} \exp \Big(2 \binom k2 \sum_{\substack{p|(h_i-h_j) \\ p > k^3}} \frac 1p \Big).
\end{equation}
\end{lemma}
\begin{proof}
By definition,
\begin{align}
\mathfrak S(\mathcal H) &= \prod_{p \text{ prime}} \frac{1-\nu_{\mathcal H}(p)/p}{(1-1/p)^k} \nonumber \\
&\le \prod_{\substack{p \text{ prime} \\ p \le k^3}} \frac 1{(1-1/p)^k} \prod_{\substack{p \text{ prime} \\ p > k^3}} \frac{1-\nu_{\mathcal H}(p)/p}{(1-1/p)^k} \nonumber \\
&= \prod_{p \le k^3} \frac{1}{(1-1/p)^k}\prod_{p > k^3} \frac{1-k/p}{(1-1/p)^k}\prod_{\substack{p|\Delta(\mathcal H) \\ p > k^3}} \frac{p-\nu_{\mathcal H}(p)}{p-k}. \label{eq:splitproductintoprimeregions}
\end{align}
Rewrite the product inside via
\begin{align*}
\prod_{\substack{p|\Delta(\mathcal H) \\ p > k^3}} \frac{p-\nu_{\mathcal H}(p)}{p-k} &\le \exp\Big(\sum_{\substack{p|\Delta(\mathcal H) \\ p > k^3}} \frac{p-\nu_{\mathcal H}(p)}{p-k}\Big) \\ 
&\ll \exp\Big(2 \sum_{1 \le i < j\le k} \sum_{\substack{p|(h_i-h_j) \\ p > k^3}} \frac 1p \Big)  \\ 
&\le \binom{k}{2}^{-1}\sum_{1 \le i < j \le k} \exp \Big(2 \binom k2 \sum_{\substack{p|(h_i-h_j) \\ p > k^3}} \frac 1p \Big),
\end{align*}
where the last step comes from Jensen's inequality; plugging this into \eqref{eq:splitproductintoprimeregions} yields the result.
\end{proof}

With Lemma \ref{lem:amgmbound} in hand, we now turn to the proof of Theorem \ref{thm:bigk:sumboundallk}.

\begin{proof}[Proof of Theorem \ref{thm:bigk:sumboundallk}]
By Lemma \ref{lem:amgmbound}, for each set $\mathcal H = \{h_1, \dots, h_k\}$ with $h_1, \dots, h_k \le h$ distinct,
\[\mathfrak S(\mathcal H) \ll \prod_{p \le k^3} \frac{1}{(1-1/p)^k}\prod_{p > k^3} \frac{1-k/p}{(1-1/p)^k} \binom{k}{2}^{-1}\sum_{1 \le i < j \le k} \exp \Big(2 \binom k2 \sum_{\substack{p|(h_i-h_j) \\ p > k^3}} \frac 1p \Big).\]
Sum over $\mathcal H$ to get
\begin{equation*}
T_k(h) \ll  \prod_{p \le k^3} \frac{1}{(1-1/p)^k}\prod_{p > k^3} \frac{1-k/p}{(1-1/p)^k} \subsum{1\le h_1, \dots, h_k \le h \\ \text{distinct}} \binom{k}{2}^{-1}\sum_{1 \le i < j \le k} \exp \Big(2 \binom k2 \sum_{\substack{p|(h_i-h_j) \\ p > k^3}} \frac 1p \Big).
\end{equation*}
Let $S_k(h)$ refer to the sums above, so that
\begin{equation*}
S_k(h) := \subsum{1\le h_1, \dots, h_k \le h \\ \text{distinct}}\binom{k}{2}^{-1} \sum_{1 \le i < j \le k} \exp \Big(2 \binom k2 \sum_{\substack{p|(h_i-h_j) \\ p > k^3}} \frac 1p \Big).
\end{equation*}
Then
\begin{align*}
S_k(h)&\ll \binom{k}{2}^{-1} \sum_{1 \le i < j \le k} \subsum{1\le h_1, \dots, h_k \le h \\ \text{distinct}} \exp \Big(2\binom k2 \sum_{\substack{p|(h_i-h_j)\\p>k^3}} \frac 1p \Big)  \\
&\ll \binom{k}{2}^{-1} \sum_{1 \le i < j \le k} \sum_{\substack{h_i,h_j \le h \\ \text{distinct}}} \exp\Big(2\binom k2 \sum_{\substack{p|(h_i-h_j) \\ p > k^3}} \frac 1p \Big)h^{k-2} \\
&= h^{k-2} \sum_{\substack{\ell \le h}} (h-\ell)\exp\Big(2\binom k2 \sum_{\substack{p|\ell \\ p > k^3}} \frac 1p \Big)\\
&\le h^{k-1} \sum_{\substack{\ell \le h}} \exp\Big(2\binom k2 \sum_{\substack{p|\ell \\ p > k^3}} \frac 1p \Big).
\end{align*}

The sum over $\ell$ is a sum over a multiplicative function $f_k(\ell)$, with $f_k(p^j) = 1$ if $p \le k^3$ and $f_k(p^j) = \exp(k(k-1)/p)$ if $p>k^3$, regardless of $j$. The function $f_k(\ell)$ satisfies
\begin{equation*}
f_k(\ell) = \sum_{d|\ell} g_k(d),
\end{equation*}
where $g_k$ is a multiplicative function given by $g_k(p^j) = 0$ if $p \le k^3 $ or $j \ge 2$ and $g_k(p) = \exp(k(k-1)/p)-1$ for $p > k^3$ prime. Then
\begin{equation*}
\sum_{\ell \le h} f_k(\ell) = \sum_{\ell \le h} \sum_{d|\ell} g_k(d) 
= \sum_{d\le h} g_k(d)\left\lfloor \frac hd\right\rfloor
\le h\sum_{d=1}^\infty \frac{g_k(d)}{d}.
\end{equation*}
The sum over $d$ can be rewritten as
\begin{align*}
\prod_{p > k^3} \Big(1 + \frac{\exp\left(\tfrac{k(k-1)}{p}\right)-1}{p}\Big) &= \exp\Big(\sum_{\substack{p \text{ prime}\\p>k^3}} \sum_{j\ge 1} \frac 1p \left(\frac{k(k-1)}{p}\right)^j\Big) \\
&< \exp\Big(\sum_{\substack{p \text{ prime} \\ p>k^3}} \sum_{j \ge 1} \frac 1{p^{1+j/3}}\Big). 
\end{align*}
The sum in the exponent is bounded by a constant independent of $k$, and thus the sum over $d$ is bounded by a constant independent of $k$, so that $S_k(h) \ll h^k$.

Finally, return to the contribution from the small primes and $T_k(h)$, which is bounded by
\begin{equation*}
T_k(h) \ll h^k \prod_{p \le k^3} \frac 1{(1-1/p)^k} \ll h^k (3\log k)^k,
\end{equation*}
as desired.
\end{proof}

\section{Proof of Theorem \ref{thm:distributiontail:momentboundsmallk} and its corollaries}\label{sec:distributiontail:HLbounds}

Throughout, consider an interval of size $h = \l \log x$. Fix $\d > \frac 12$ and assume that $r \ll (\log h)^{1-\d}$. We begin with the proof of Theorem \ref{thm:distributiontail:momentboundsmallk}.

The $r$th moment $m_r(x,h)$, defined in \eqref{eq:rmomentlogdef}, is given by
\begin{align*}
m_r(x,h) &= \subsum{1 \le h_1, \dots, h_r \le h} \frac 1x \sum_{\substack{n \le x}} \mathbf 1_{\mathcal P}(n+h_1) \cdots \mathbf 1_{\mathcal P}(n+h_r) \\
&= \sum_{\ell=1}^r \frac{\sigma(r,\ell)}{\ell!} \subsum{1 \le h_1, \dots, h_\ell \le h \\ \text{distinct}} \frac 1x \sum_{n \le x} \mathbf 1_{\mathcal P}(n+h_1) \cdots \mathbf 1_{\mathcal P}(n+h_\ell),
\end{align*}
where $\sigma(r,\ell)$ is the number of surjective maps $[1,r] \onto [1,\ell]$, and $\frac{\sigma(r,\ell)}{\ell!} = \stirlingii{r}{\ell}$, the Stirling number of the second kind.

Apply Conjecture \ref{conj:bg:HLuniform} to replace the sum over correlations of primes with a sum over singular series, yielding
\begin{equation*}
m_r(x,h) = \sum_{\ell=1}^r \stirlingii{r}{\ell} \frac 1{(\log x)^\ell} \subsum{1 \le h_1, \cdots, h_\ell \le h \\ \text{distinct}} (\mathfrak S(\{h_1, \dots, h_\ell\}) + o(1)).
\end{equation*}
Estimating this moment now depends on the average of the singular series constants, and in particular how quickly this average converges to $1$. We apply our results from Section \ref{sec:bigksums:gallagher} bounding sums of singular series for large sets, and in particular Theorem \ref{thm:bigksums:extendinggallagher}, which requires our assumption that $r \ll (\log h)^{1-\d}$. For larger $r$, one could also apply the weaker result in Theorem \ref{thm:bigk:sumboundallk} to yield a weaker moment bound.

By Theorem \ref{thm:bigksums:extendinggallagher}, for any $\ell \le r \ll (\log h)^{1-\d}$ and for some $\b > 0$ dependent only on $\d > \frac 12$, 
\begin{equation*}
\sum_{\substack{1 \le h_1, \cdots, h_\ell \le h \\ \text{distinct}}} \mathfrak S(h_1, \dots, h_\ell) =  h^\ell + O(h^{\ell-\beta}).
\end{equation*}
Then for any $r \ll (\log h)^{1-\d}$,
\begin{align*}
m_r(x,h) &= \sum_{\ell = 1}^r \stirlingii{r}{\ell} \frac 1{(\log x)^\ell}\left(h^\ell  +  O(h^{\ell-\b})\right) + o\left(\sum_{\ell = 1}^r \frac 1{(\log x)^\ell}\stirlingii{r}{\ell} h^\ell\right) \\
&=  \left(\sum_{\ell = 1}^r \stirlingii{r}{\ell} \l^\ell\right)(1 + o(1)),
\end{align*}
where the error term is uniform in $r$. This completes the proof of Theorem \ref{thm:distributiontail:momentboundsmallk}. We now proceed to prove Corollary \ref{cor:distributiontail:smallktailbound}.

\begin{proof}[Proof of Corollary \ref{cor:distributiontail:smallktailbound}]
Let $r \ll (\log h)^{1-\d}$. Applying a Markov bound to the $r$th moment $m_r(x,h)$, we get that
\begin{equation*}
I(x;k,h) \le \frac{x}{k^r} m_r(x,h)
\ll \frac{x}{k^r} \sum_{\ell = 1}^r \stirlingii{r}{\ell}\l^\ell.
\end{equation*}
As shown in \cite[Theorem 3]{MR241310}, Stirling numbers of the second kind are bounded above by $\stirlingii{r}{\ell} \le \frac 12 \binom r{\ell} \ell^{r-\ell},$ so
\begin{equation*}
I(x;k,h) \ll \frac{x}{k^r} \sum_{\ell = 1}^r  \binom r{\ell} \ell^{r-\ell} \l^\ell \ll \frac{x}{k^r} (\l + r)^r. 
\end{equation*}
If $\l \ge 1$, then as $x \to \infty$ eventually $\frac{k}{\l} \ge \frac{\l e}{\l - 1}$, which in turn implies that we can choose $r = \frac{k}{\l e}$ and get $(\l + r)^r \le (\l r)^r$. With this choice of $r$, we thus get $I(x;k,h)\ll xe^{-k/\l e}$, as desired.

Meanwhile if $\l < 1$, we can choose $r = \frac{k}{(\l + 1)e}$ to get the desired result, since $(\lambda+r)^r \le ((\lambda+1)r)^r$.
\end{proof}

Corollary \ref{cor:distributiontail:biggerktailbound} follows via the same argument as the proof of Corollary \ref{cor:distributiontail:smallktailbound}, but where $r$ is taken to be $(\log h)^{1-\d}$.

\section{Unconditional bounds}\label{sec:distributiontail:unconditionalbounds} 

We can also achieve weaker unconditional bounds on the moments $m_r(x, h)$ and the tail of the distribution via replacing the use of the Hardy--Littlewood conjectures by an application of the Selberg sieve. More precisely, we will make use of the following theorem, which is proven in Section \ref{sec:distributiontail:selbergsievethm}.
\begin{theorem}\label{thm:distributiontail:selbergsievethm}
Let $x \ge 2$, let $k = o((\log x)^{1/4})$, and let $\mathcal H = \{h_1, \dots, h_k\}$ be a set of $k$ distinct natural numbers. For any $\ep > 0$, 
\begin{align*}
&\left|\left\{n \le x: n + h_i \text{ prime for all } i \right\}\right| \\
&\le (2+\ep)^k k! \mathfrak S(\mcH) \frac{x}{\log^k x} \left(1 + O\left(\frac{\log\log(3x) + k^4 + k\log\log(3|D_{\mcH}|)}{\log x}\right) \right),
\end{align*}
where $D_{\mcH} := \prod_{i < j} (h_i - h_j)$. 
\end{theorem}
Theorem \ref{thm:distributiontail:selbergsievethm} extends the work of Klimov in \cite{MR0097372}, who shows an analogous bound for $k$ fixed as $x \to \infty$. 

We now turn to the proofs of Theorem \ref{thm:distributiontail:unconditionalmomentbound}, as well as that of Corollary \ref{cor:distributiontailuncondtionalbound}.
\begin{proof}[Proof of Theorem \ref{thm:distributiontail:unconditionalmomentbound}]
As in the proof of Theorem \ref{thm:distributiontail:momentboundsmallk}, we have
\begin{equation*}
m_r(x,h) = \sum_{\ell=1}^r \frac{\sigma(r,\ell)}{\ell!} \subsum{1 \le h_1, \cdots, h_\ell \le h \\ \text{distinct}} \frac 1x \sum_{n \le x} \mathbf 1_{\mathcal P}(n+h_1) \cdots \mathbf 1_{\mathcal P}(n+h_\ell).
\end{equation*}
For our choice of $h$ and $r$, the error term in Theorem \ref{thm:distributiontail:selbergsievethm} is $O(1)$. Applying Theorem \ref{thm:distributiontail:selbergsievethm}, the $r$th moment is then bounded by
\begin{align*}
m_r(x,h) &\ll \sum_{\ell=1}^r \frac{\sigma(r,\ell)}{\ell!(\log x)^\ell} \subsum{1 \le h_1, \cdots, h_\ell \le h \\ \text{distinct}} (2+\ep)^\ell \ell! \mathfrak S(\mcH) \\
&\ll \sum_{\ell=1}^r \stirlingii{r}{\ell} \ell! \frac{1}{(\log x)^\ell} (2+\ep)^\ell  h^\ell e^{O(\ell \log \log \ell)},
\end{align*}
where the last step follows by applying Theorem \ref{thm:bigk:sumboundallk}. Since $h = \l\log x$, this sum is then
\begin{align*}
&\ll r! (2 + \ep)^r e^{O(r\log\log r)} \sum_{\ell = 1}^r \stirlingii{r}{\ell} \lambda^\ell.
\end{align*}
As seen in the proof of Corollary \ref{cor:distributiontail:smallktailbound}, $\sum_{\ell = 1}^r \stirlingii{r}{\ell} \lambda^\ell \le (\lambda+r)^r\le ((\l+1)r)^r,$ so that
\begin{align*}
m_r(x,h) &\ll r! (2+\ep)^r e^{O(r\log \log r)}(\l+1)^r r^r \\
&\ll r^{2r} e^{O(r\log \log r)} (\l+1)^r,
\end{align*}
which gives the result.
\end{proof}

\begin{proof}[Proof of Corollary \ref{cor:distributiontailuncondtionalbound}]
The proof of this corollary proceeds along the same lines as the proofs of Corollaries \ref{cor:distributiontail:smallktailbound} and \ref{cor:distributiontail:biggerktailbound}. In this case, we know unconditionally from Theorem \ref{thm:distributiontail:unconditionalmomentbound} that 
\begin{equation*}
I(x;k,h) \le \frac{x}{k^r} m_r(x,h) \ll \frac{x}{k^r}(\l+1)^r r^{2r} e^{O(r\log\log r)}.
\end{equation*}
Hence there exists some constant $C > 0$ with 
\begin{equation*}
I(x;k,h) \ll \frac{x}{k^r} (\l+1)^r r^{2r} e^{Cr\log\log r} = x\exp(r\log \frac{\l+1}{k} + 2r\log r + Cr\log\log r).
\end{equation*}
Choose $r = \left(\frac{k}{(\l+1) e}\right)^{1/2}2^{C/2}\left(\log \frac{k}{(\l+1) e}\right)^{-C/2}$, so that 
\begin{align*}
\log r &= \frac 12 \log \frac{k}{(\l +1)e} +\frac C2 \log 2- \frac C2 \log \log  \frac{k}{(\l+1)e}, \text{ and} \\
\log \log r &= \log \frac 12 + \log (\log \frac{k}{(\l+1)e} - \frac C2 \log \log \frac{k}{(\l +1)e} + C\log 2) \\
&\le \log \frac 12 + \log \log \frac{k}{(\l + 1)e},
\end{align*}
where the inequality holds for large enough $x$ since $\frac{k}{\l} \to \infty$. Plugging in these expressions for $\log r$ and $\log \log r$ gives
\begin{align*}
r&\log \frac{\l +1}{k} + 2r\log r + Cr\log \log r \le -r,
\end{align*}
so that $I(x;k,h) \ll x e^{-r}$, which completes the proof.
\end{proof}

\subsection{Selberg's Sieve: Proof of Theorem \ref{thm:distributiontail:selbergsievethm}} \label{sec:distributiontail:selbergsievethm}

Selberg's sieve has previously been used to bound the frequency of prime $k$-tuples; see for example \cite{MR0424730}, which we will refer to throughout this section, as well as \cite{MR2647984} and \cite{MR0097372}. In \cite{MR0424730}, Halberstam and Richert proceed along a very similar calculation, with the only material difference being that we are not taking $k$ to be a constant in terms of the other parameters, and thus we keep track of the dependence on $k$ throughout. We proceed along the lines of \cite[Theorem 5.7]{MR0424730}. To do so, there are several lemmas that we will want to adapt to this setting.

We begin by defining notation. Let $\mathcal P$ be the set of all primes, and for $z > 0$, let $P(z) := \prod_{p \le z} p$. Let $\mathcal H = \{h_1, \dots, h_k\}$ be a set of $k$ distinct natural numbers, so that $D_{\mcH} \ne 0$.
Define
\begin{equation*}
\mathcal A := \left\{ \prod_{i=1}^k (n + h_i) : n \le x \right\},
\end{equation*}
and define $A_p$ to be the number of elements of $\mathcal A$ that are divisible by a prime $p$, with $A_p = \frac{\nu_{\mcH}(p)}{p}x + O(\nu_{\mcH}(p)).$ Let $A_d$ be the number of elements of $\mathcal A$ that are divisible by $d$, so that $A_d = \frac{\nu_{\mcH}(d)}{d} x + O(\nu_{\mcH}(d)),$ where $\nu_{\mcH}(d) = \prod_{p|d} \nu_{\mcH}(p)$. Let $R_d = A_d - \frac{\nu_{\mcH}(d)}{d} x$, so that $|R_d| \le \nu_{\mcH}(d)$. Our goal is to estimate the quantity
\begin{equation}\label{eq:distributiontail:selberg:SAPzdef}
S(\mathcal A;\mathcal P,z):=|\{a:a \in \mathcal A, (a,P(z)) = 1\}|.
\end{equation}

Halberstam and Richert define three conditions on a sieve problem in order to apply the Selberg sieve, which they denote ($R$), ($\Omega_1$), and $(\Omega_2(\kappa,L))$, where the parameter $\kappa$ is the dimension of the sieve, which in this case is equal to $k$. The three conditions are:
\begin{align}
|R_d| &\le \nu_{\mcH}(d) \text{ if } \mu(d) \ne 0; \tag{$R$} \\
0 &\le \frac{\nu_{\mcH}(p)}{p} \le 1 - \frac 1{\a_1} \text{ for some constant $\a_1 \ge 1$}; \tag{$\Omega_1$} \\
-L &\le \sum_{w \le p < z} \frac{\nu_{\mcH}(p) \log p}{p} - \kappa \log \frac zw \le \a_2 \text{ for any } z\ge w \ge 2. \tag{$\Omega_2(\kappa,L)$}
\end{align}
For the final condition, $\a_2$ and $L$ are constants, each $\ge 1$, which are independent of $z$ and $w$. For our purposes, we will need to keep track of the values $\a_1$, $\a_2$, and $L$, and in particular their dependence on $k$.

\begin{lemma}\label{lem:conditions-r-omegas-hold}
For a set $\mathcal H = \{h_1, \dots, h_k\}$ and $\nu_{\mcH}$, $D_{\mcH}$, $\mathcal A$, $A_d$, and $R_d$ defined as above, the conditions ($R$), ($\Omega_1$), and ($\Omega_2(\kappa,L)$) are satisfied with $\a_1 = k +1$, $\a_2 = O(k)$, $\kappa = k$, and $L = k \log\log(3|D_{\mcH}|)$.
\end{lemma}
\begin{proof}
We first see that condition ($R$) is satisfied, since $|R_d| \le \nu_{\mcH}(d)$. We also have $\frac{\nu_{\mcH}(p)}{p} \le \frac{\min\{k,p-1\}}{p} \le 1-\frac{1}{k+1},$ so that ($\Omega_1$) is satisfied with $\a_1 = k+1$.

For any $z$ and any $w < z$, 
\[\sum_{w \le p < z} \frac{\nu_{\mcH}(p)\log p}{p} = k \sum_{w \le p < z} \frac{\log p}{p} - \sum_{w \le p < z} \frac{k-\nu_{\mcH}(p)}{p} \log p,\]
so that
\[\sum_{w \le p < z} \frac{\nu_{\mcH}(p) \log p}{p} \le k \sum_{w \le p < z} \frac{\log p}{p} \le k \log \frac zw + O(k).\]
By \cite[Lemma 5.1]{MR0424730}, for any natural number $n$, $\sum_{p|n} \frac{\log p}{p} \ll \log \log(3n)$, so 
\[\sum_{w \le p < z} \frac{\nu_{\mcH}(p) \log p}{p} \ge k \log \frac zw - O(k) - \sum_{p|D_{\mcH}} \frac{k}{p} \log p \ge k \log \frac zw - O(k) - O(k\log\log(3|D_{\mcH}|)),\]
and thus ($\Omega_2(\kappa,L)$) is satisfied with $\kappa = k$, $\a_2 = O(k)$ and $L = O(k\log\log(3|D_{\mcH}|))$.
\end{proof}

For $d$ squarefree, let $g(d) = \frac{\nu_{\mcH}(d)}{d \prod_{p|d}(1-\nu_{\mcH}(p)/p)}$, let $G(z) = \sum_{d < z} \mu(d)^2g(d),$ and more generally for any $x$ define $G(x,z) = \sum_{\substack{d < x \\ d|P(z)}} \mu(d)^2 g(d)$, so that $G(z) = G(z,z)$. Let $W(z) = \prod_{p<z}\left(1-\frac{\nu_{\mcH}(p)}{p}\right).$ 
\begin{lemma}[After Lemma 5.4 from \cite{MR0424730}]\label{lem:distributiontail:selberg:HR5.4}
Fix a set $\mathcal H = \{h_1, \dots, h_k\}$, and define $\nu_{\mcH}$, $g(d)$ $G(x,z),$ and $W(z)$. Let $L = k\log\log(3|D_{\mcH}|)$, and let $z > 0$ be a number such that for sufficiently large constants $B_L$ and $B_k$, $L \le \frac 1{B_L} \log z$ and $k^2 \le \frac 1{B_k} \log z$. Then
\begin{equation*}
\frac 1{G(z)} = W(z)e^{\gamma k} \Gamma(k+1)\left(1 + O\left(\frac{L + k^4}{\log z}\right)\right).
\end{equation*}
\end{lemma}

\begin{proof}
This proof closely follows the proof of \cite[Lemma 5.4]{MR0424730}, so here we simply highlight the differences, which arise only in that here we keep track of the dependence on $k$ in the form of the constants $\a_1, \a_2$, $\kappa$, and $L$. By Lemma \ref{lem:conditions-r-omegas-hold}, conditions ($R$), ($\Omega_1$), and ($\Omega_2(k,L)$) hold.

In \cite{MR0424730}, Halberstam and Richert show that for $0 < z \le x$,
\begin{align*}
\subsum{d < x \\ d|P(z)} g(d) \log d = &\subsum{d < x \\ d|P(z)} g(d) \sum_{p < \min\{x/d,z\}} \frac{\nu_{\mcH}(p)}{p} \log p \\
&+ \subsum{xz^{-2} \le d < x \\ d|P(z)} g(d) \subsum{\sqrt{x/d} \le p < \min\{x/d,z\} \\ p \nmid d} \frac{g(p)\nu_{\mcH}(p)}{p} \log p.
\end{align*}
Halberstam and Richert use ($\Omega_2(k,L)$) to evaluate the first inner sum. For the second inner sum, note that $\frac{\nu_{\mcH}(p)}{p} \log p \le \a_2$ (see \cite[Equation (2.3.8)]{MR0424730}), and by \cite[Equation (2.3.11)]{MR0424730} we have for any $2 \le a \le b$ that
\begin{equation}\label{eq:sum-g-over-primes-bound}
\sum_{a \le p < b} g(p) \le k \log \frac{\log b}{\log a} + \frac{\a_2}{\log a} + \frac{\a_1\a_2}{\log a}\left(k+\frac{\a_2}{\log a}\right) = k \log \frac{\log b}{\log a} + O\left(\frac{k^3}{\log a}\right),
\end{equation}
which implies that
\begin{align*}
\subsum{\sqrt{x/d} \le p < \min\{x/d,z\} \\ p \nmid d} \frac{g(p)\nu_{\mcH}(p)}{p} \log p &\le \a_2 \subsum{\sqrt{x/d} \le p < \min\{x/d,z\} \\ p \nmid d}g(p) \\
&\ll \a_2k + \a_2^2 + \a_1\a_2^2k + \a_1\a_2^3 \ll k^4.
\end{align*}
Combining these estimates, we get
\begin{align*}
\subsum{d < x \\ d|P(z)} g(d) \log d = &\subsum{x/z \le d < x \\ d|P(z)} g(d) \left(k\log \frac xd + O(L)\right) \\
&+ \subsum{d<x/z \\ d|P(z)} g(d) (k \log z + O(L)) + O(k^4G(x,z)) \\
&= k \subsum{d<x \\ d|P(z)} g(d) \log \frac xd - k \subsum{d<x/z \\ d|P(z)} g(d) \log \frac{x/z}{d} + O((L + k^4)G(x,z)).
\end{align*}

This expression is identical to what appears in \cite[Page 149]{MR0424730} in the proof of Lemma 5.4, except that the factor of $L$ in the error term is replaced by a factor of $L + k^4$. The rest of the proof applies to our situation without change, except for replacing factors of $L$ in the error term with $L+k^4$, so that, as in \cite[Equations (3.10) and (3.13)]{MR0424730}, we get
\begin{equation*}
G(z) = \frac 1{\mathfrak S(\mcH)\Gamma(k+1)} (\log z)^k \left(1+ O\left(\frac{L+k^4}{\log z}\right)\right).
\end{equation*}

By following the proof of \cite[Lemma 5.2]{MR0424730} and using \eqref{eq:sum-g-over-primes-bound} in place of \cite[Equation (2.3.4)]{MR0424730}), we get that for some constant $C \in \R$,
\begin{equation}\label{eq:distributiontail:lemma5.2hrbound}
C\frac{L}{\log a} \le \sum_{a \le p < b} \frac{\nu_{\mcH}(p)}{p}-\kappa \sum_{a \le p < b} \frac 1p \le O\left(\frac{k^3}{\log a}\right).
\end{equation}
Note that by ($\Omega_1$) and \cite[Equation (2.3.9)]{MR0424730}, we have $\sum_{a \le p < b} g^2(p) = O(k^4/\log a)$. Thus, by following the proof of \cite[Lemma 5.3]{MR0424730} and using \eqref{eq:distributiontail:lemma5.2hrbound}, we get that
\begin{equation*}
\prod_{p \ge z} \left(1 - \frac{\nu_{\mcH}(p)}{p}\right)^{-1}\left(1-\frac 1p\right)^k = 1 + O\left(\frac{L+k^4}{\log z}\right).
\end{equation*}
This implies that
\begin{equation}\label{eq:distributiontail:Wzestimate}
W(z) = \mathfrak S(\mathcal H) \frac{e^{-\g k}}{(\log z)^k}\left( 1+O\left(\frac{L+k^4}{\log z}\right)\right),
\end{equation}
which corresponds to \cite[Equation (5.2.5)]{MR0424730} and completes the proof.
\end{proof}

We also make use of \cite[Theorem 3.1]{MR0424730}, which we cite without modification.
\begin{theorem}[Theorem 3.1, Halberstam--Richert, \cite{MR0424730}]\label{thm:distributiontai:hrthm3.1}
Using the notation of this section, with $S(\mathcal A; \mathcal P,z)$ defined in \eqref{eq:distributiontail:selberg:SAPzdef} and satisfying ($R$) and ($\Omega_1$), we have
\begin{equation*}
S(\mathcal A;\mathcal P,z) \le \frac{x}{G(z)} + \frac{z^2}{W^3(z)}.
\end{equation*}
\end{theorem}

We are now ready to prove Theorem \ref{thm:distributiontail:selbergsievethm}, following the proof of \cite[Theorem 5.7]{MR0424730}. Fix $z < x$ to be chosen later; we estimate $S(\mathcal A;\mathcal P,z)$, which is an upper bound for our desired quantity. By Theorem \ref{thm:distributiontai:hrthm3.1}, 
\begin{equation*}
S(\mathcal A;\mathcal P,z) \le \frac{x}{G(z)} + \frac{z^2}{W^3(z)}.
\end{equation*}
Applying Lemma \ref{lem:distributiontail:selberg:HR5.4} yields
\begin{equation*}
S(\mathcal A;\mathcal P,z) \le xW(z)e^{\g k}\Gamma(k+1)\left(1+O\left(\frac{L+k^4}{\log z}\right)\right) + x\frac{z^2}{xW^3(z)}.
\end{equation*}
By equation (2.3.12) in \cite{MR0424730},
\begin{equation*}
\frac 1{W(z)} \ll e^{O(k^3)}(\log z)^k,
\end{equation*}
which implies that
\begin{equation*}
\frac{z^2}{W^3(z)} = xW(z)\frac{z^2}{xW^4(z)} = xW(z)O\left(\frac{z^2(\log z)^{4k}e^{O(k^3)}}{x}\right),
\end{equation*}
and thus
\begin{equation*}
S(\mathcal A;\mathcal P,z) \le xW(z)\left(e^{\g k}\Gamma(k+1)\left(1+O\left(\frac{L+k^4}{\log z}\right)\right) + O\left(\frac{z^2(\log z)^{4k}e^{O(k^3)}}{x}\right)\right).
\end{equation*}

Plugging in \eqref{eq:distributiontail:Wzestimate}, we get
\begin{equation*}
S(\mathcal A;\mathcal P,z) \le \Gamma(k+1) \mathfrak S(\mathcal H) \frac{x}{(\log z)^k} \left(1 +O\left(\frac{L+k^4}{\log z}\right) + O\left(\frac{z^2(\log z)^{4k}e^{O(k^3)}}{x}\right)\right).
\end{equation*}
Set $z = x^{1/(2 +\ep)}$ to complete the proof, keeping in mind that $k = o((\log x)^{1/4})$. 

\bibliographystyle{amsplain}
\bibliography{singseries}

\end{document}